\documentclass[12pt]{article}
 
 \usepackage{amsmath,amssymb,amscd,amsthm,esint}
 
\usepackage{graphics,amsmath,amssymb,amsthm,mathrsfs}

\usepackage{graphics,amsmath,amssymb,amsthm,mathrsfs,amsfonts}

\usepackage{sidecap}
\usepackage{float}
\usepackage{extarrows}
\usepackage{booktabs}
\usepackage{verbatim}
\usepackage{hyperref}
\usepackage[usenames,dvipsnames]{xcolor}

\newtheorem{thm}{Theorem}[section]
\newtheorem{lemma}[thm]{Lemma}

\theoremstyle{definition}
\newtheorem{remark}[thm]{Remark}

\def\XXint#1#2#3{{\setbox0=\hbox{$#1{#2#3}{\int}$}
         \vcenter{\hbox{$#2#3$}}\kern-.5\wd0}}

\def\R{\mathbb{R}}

\def\e{\varepsilon}
\def\wN{\widetilde{N}}

\numberwithin{equation}{section}

\begin{document}

\title{Nodal Sets and Doubling Conditions\\ in Elliptic Homogenization}

\author{Fanghua Lin \thanks{Supported in part by NSF grant DMS-1501000.}
\qquad
Zhongwei Shen\thanks{Supported in part by NSF grant DMS-1600520.}}
\date{}
\maketitle

\centerline{Dedicated to Our Teacher Professor Carlos E. Kenig}
\centerline{on the Occasion of His 65th Birthday}

\begin{abstract}

This paper is concerned with uniform measure estimates for nodal sets of solutions
in elliptic homogenization. 
We consider a family of second-order elliptic operators $\{ \mathcal{L}_\e\}$
in divergence form with rapidly oscillating and periodic coefficients.
We show that the $(d-1)$-dimensional Hausdorff measures of
the nodal sets of solutions to $\mathcal{L}_\e (u_\e)=0$ in a ball in $\R^d$
are bounded  uniformly  in $\e>0$.
The proof relies on a uniform doubling condition and approximation of $u_\e$
by solutions of the homogenized equation.

\end{abstract}

\section{\bf Introduction}\label{section-1}

In this paper we initiate the study of uniform measure estimates for nodal sets of solutions
in elliptic homogenization.
We consider a family of elliptic operators in divergence form,
\begin{equation}\label{operator}
\mathcal{L}_\e =-\text{\rm div} \big(A(x/\e)\nabla\big),
\end{equation}
where $\e>0$ and $A(y)=(a_{ij} (y))$ is a symmetric $d\times d$ matrix-valued function
in $\R^d$.
Throughout this paper, unless otherwise stated, we will impose the following conditions,
\begin{itemize}

\item (ellipticity) there exists some $\lambda\in (0, 1]$ such that
\begin{equation}\label{ellipticity}
\lambda |\xi|^2 \le \langle A(y)\xi, \xi\rangle  \le |\xi|^2 \quad \text{ for any } y, \xi \in \R^d;
\end{equation}
 
 \item 
 
 (periodicity) $A$ is 1-periodic,
 \begin{equation}\label{periodicity}
 A(y+z)=A (y) \quad \text{ for any } y\in \R^d \text{ and } z\in \mathbb{Z}^d.
 \end{equation}
 
 \end{itemize}
Our main results, as in the case $\e=1$,
 also require the following Lipschitz condition:
 there exists some $M>0$ such that
 \begin{equation}\label{smoothness}
 |A(x)-A(y)|\le M |x-y| \quad \text{ for any } x, y\in \R^d.
 \end{equation}

 Let $B(x, r)=\{ y\in \R^d: |y-x|<r \}$ and $B_r =B(0,r)$.
 
 \begin{thm}\label{main-theorem-1}
 Assume that $A=A(y)$ satisfies conditions (\ref{ellipticity}), (\ref{periodicity}) and
 (\ref{smoothness}).
 Let $u_\e\in H^1(B_2)$ be a nonzero  weak solution of $\mathcal{L}_\e (u_\e)=0$ in $B_2$.
 Suppose that
 \begin{equation}\label{double-1}
 \fint_{B_2} |u_\e|^2\, dx \le N \fint_{B_{\sqrt{\lambda}}} |u_\e|^2\, dx
 \end{equation}
 for some $N>1$.
 Then
 \begin{equation}\label{main-estimate-1}
\mathcal{H}^{d-1}
\big\{ x\in B_{\sqrt{\lambda}/4}: \ u_\e (x)=0 \big\}
\le C(N),
\end{equation}
where $C(N)$ depends only on $d$, $\lambda$, $M$ and $N$.
 \end{thm}
 
 The study of nodal sets for solutions and eigenfunctions 
 is important  for understanding geometric properties of elliptic operators.
 Classical results in this area may be found
 in \cite{D-Fefferman, Hardt-Simon, Lin-N, Han-Lin-1, Han-Lin-2, Han-1994}.
 See \cite{Logunov-2018-1, Logunov-2018-2} and their references for more recent advances.
 In particular, for $\e=1$, Theorem \ref{main-theorem-1}
 was proved  in \cite{Han-Lin-2} by Q. Han and F. Lin.
Since the constants $C$ depend on the smoothness of coefficients,
quantitative results in \cite{Han-Lin-2} as well as in other previous work 
do not extend directly to the operator $\mathcal{L}_\e$ for estimates that are uniform with respect to $\e$.
Our Theorem \ref{main-theorem-1} provides the first result on the uniform measure estimates
of nodal sets of solutions for $\mathcal{L}_\e$ in the periodic setting.

Our general approach to the estimate (\ref{main-estimate-1})
follows the iterating-rescaling scheme used in \cite{Han-Lin-2} 
(also see related earlier work in \cite{Hardt-Simon, Lin-N}).
As in \cite{Han-Lin-2}, the proof relies on the doubling condition for solutions.
 
 \begin{thm}\label{main-theorem-2}
 Assume that $A=A(y)$ satisfies conditions (\ref{ellipticity}), (\ref{periodicity}) and
 (\ref{smoothness}).
 Let $u_\e\in H^1(B_2)$ be a weak solution of $\mathcal{L}_\e (u_\e)=0$ in $B_2$.
 Suppose $u_\e$ satisfies (\ref{double-1}) for some $N>1$.
 Then
 \begin{equation}\label{main-estimate-2}
 \fint_{B_r} |u_\e|^2\, dx
 \le C (N)\fint_{B_{r/2}} |u_\e|^2\, dx
 \end{equation}
 for any $0<r<1$,
 where $C(N)$ depends only on $d$, $\lambda$, $M$ and $N$.
 \end{thm}
 
 The doubling condition for $\mathcal{L}_1$, as a consequence of a monotonicity 
 formula, was proved in \cite{G-Lin}.
 The proof of Theorem \ref{main-theorem-2} for $\mathcal{L}_\e$
 uses a compactness argument from the theory of periodic  homogenization.
 The idea is that as $\e\to 0$,
 $u_\e$ converges strongly in $L^2$ to a solution $u_0$ of a second-order elliptic
 equation with constant coefficients.
 Together with a three-spheres theorem for $u_0$, this yields (\ref{main-estimate-2}) for some small $r>0$.
 By an iteration argument  we then obtain  (\ref{main-estimate-2}) for $C\e<r<1$.
 Finally, the small-scale case $0<r\le C\e$ is handled by a blow-up argument.
 
 The second key ingredient in the proof of Theorem \ref{main-theorem-1}
 is an approximation result. It allows us to utilize  the existing 
 estimates of nodal and singular sets for the homogenized operator $\mathcal{L}_0$.
 
 \begin{thm}\label{main-theorem-3}
 Assume that $A=A(y)$ satisfies conditions (\ref{ellipticity}), (\ref{periodicity}) and
 (\ref{smoothness}).
Let $u_\e\in H^1(B_2)$ be a weak solution of $\mathcal{L}_\e (u_\e)=0$ in $B_2$.
Then there exists $u_0\in H^1(B_1)$ such that
$\mathcal{L}_0(u_0)=0$ in $B_1$,
\begin{equation}\label{conv-rate-1}
 \| u_\e -u_0\|_{L^\infty(B_{3/4})}
 \le C \e \, \| u_\e\|_{L^2(B_{3/2})},
 \end{equation}
 and 
 \begin{equation}\label{2-30}
 \| u_0\|_{C^1(B_1)} \le C\, \| u_\e\|_{L^2(B_{3/2})},
 \end{equation}
 where $C$ depends only on $d$, $\lambda$ and $M$.
 Moreover, if $u_\e$ satisfies (\ref{double-1}) for some $N>1$ and
 $0<\e< \e_0$, then
 \begin{equation}\label{double-0}
 \fint_{B_1} |u_0|^2\, dx \le C(N)\fint_{B_{1/2}} |u_0|^2\, dx
 \end{equation}
 where $C(N)$ and $\e_0$ depend only on $d$, $\lambda$, $M$ and $N$.
 \end{thm}
 
 The paper is organized as follows.
 In Section 2 we provide a brief review of the homogenization theory for $\mathcal{L}_\e$ and 
 give the proof of Theorem  \ref{main-theorem-3}.
  The proof of Theorem \ref{main-theorem-2} is given in Section 3,
  while Theorem \ref{main-theorem-1} is proved in Section 4.
 
 Throughout the paper we will use $C$ to denote constants that may depend on 
 $d$, $\lambda$ and $M$.
 If a constant also depends on $N$, it will be denoted by $C(N)$.
 The summation convention that repeated indices are summed will be used.

 
 \section{\bf Approximation of solutions}
 
 Suppose that $A=A(y)$ is real, bounded measurable, and satisfies the ellipticity condition 
 \begin{equation}\label{e-1}
 \lambda |\xi|^2 \le \langle A(y) \xi , \xi \rangle \quad \text{ for any } \xi\in \R^d \text{ and a.e. } y\in \R^d,
 \end{equation}
 where $\lambda>0$.
 Also assume that $A$ satisfies the periodicity condition (\ref{periodicity}).
 Let $\chi (y)=(\chi_1(y), \dots,\chi_d (y)) \in H^1(\mathbb{T}^d; \R^d)$ denote the corrector
 for $\mathcal{L}_\e$, where $\mathbb{T}^d=\mathbb{R}^d/\mathbb{Z}^d$ and
 $\chi_j$ is the unique 1-periodic function in $H^1(\mathbb{T}^d)$ such that
 \begin{equation}\label{corrector}
 \left\{
 \aligned
 &  \mathcal{L}_1 (\chi_j ) =-\mathcal{L}_1 (y_j) \quad \text{ in } \mathbb{R}^d,\\
 & \int_{\mathbb{T}^d} \chi_j \, dy  =0.
 \endaligned
 \right.
 \end{equation}
 By the classical De Giorgi - Nash estimate, $\chi_j$ is H\"older continuous.
 Moreover, $\nabla \chi$ is bounded if $A$ is H\"older continuous.
 The homogenized operator for $\mathcal{L}_\e$ is given by
 $\mathcal{L}_0 =-\text{\rm div}(\widehat{A}\nabla )$, where
 $\widehat{A}=\big(\widehat{a}_{ij} \big)_{d\times d}$ and
 \begin{equation}\label{gomo-c}
 \widehat{a}_{ij}=\fint_{\mathbb{T}^d} \left\{ a_{ij} + a_{ik} \frac{\partial \chi_j}{\partial y_k} \right\} dy.
 \end{equation}
 It is known that the homogenized matrix $\widehat{A}$ also satisfies (\ref{e-1}) with the same $\lambda$.
 Moreover, if $A$ is symmetric and satisfies (\ref{ellipticity}), 
  the same is true for $\widehat{A}$. 
  We refer the reader to \cite{JKO-1994} for the proofs.
  
  Let 
  \begin{equation}\label{B}
  B(y)=A(y) +A(y)\nabla \chi (y) -\widehat{A};
  \end{equation}
  that is $B(y)=(b_{ij} (y))_{d\times d}$ with
  $$
  b_{ij}=a_{ij} + a_{ik} \frac{\partial \chi_j}{\partial y_k}-\widehat{a}_{ij}.
  $$
  Observe that $B$ is 1-periodic and by (\ref{gomo-c}) and (\ref{corrector}),
  $$
  \int_{\mathbb{T}^d} b_{ij} \, dy =0 \quad \text{ and } \quad \frac{\partial b_{ij}}{\partial y_i}=0.
  $$
  
  \begin{lemma}\label{flux-lemma}
  There exist 1-periodic functions $\phi_{kij} (y)$ in $H^1(\mathbb{T}^d)\cap L^\infty(\mathbb{T}^d)$,
  where $1\le i, j, k\le d$, such that
  \begin{equation}\label{flux-corrector}
  b_{ij}=\frac{\partial }{\partial y_k} \phi_{kij}
  \quad \text{ and } \quad
  \phi_{kij}=-\phi_{ikj}.
  \end{equation}
  \end{lemma}
  
  \begin{proof}
  See e.g. \cite[p.1015]{KLS2}.
  \end{proof}
  
 The function $\phi =(\phi_{kij})$ is called the flux corrector for $\mathcal{L}_\e$.
  
  \begin{lemma}\label{lemma-2.1}
 Let $u_\e\in H^1(\Omega)$ and $u_0\in H^2(\Omega)$.
 Suppose that $\mathcal{L}_\e (u_\e)=\mathcal{L}_0 (u_0)$ in $\Omega$.
 Then
 \begin{equation}\label{dual-f}
 \aligned
&  \mathcal{L}_\e \left\{ u_\e -u_0 -\e \chi_j  (x/\e)\frac{\partial u_0}{\partial x_j} \right\}\\
& =-\e \frac{\partial}{\partial x_i} \left\{ \phi_{kij}(x/\e) \frac{\partial^2 u_0}{\partial x_k\partial x_j} \right\}
 +\e \text{\rm div} \left( \chi_j(x/\e) A(x/\e) \nabla \frac{\partial u_0}{\partial x_j} \right),
 \endaligned
 \end{equation}
 where $\phi=(\phi_{kij})$ is given by Lemma \ref{flux-lemma}
  \end{lemma}
  
  \begin{proof}
  See \cite[p.1016]{KLS2}.
  \end{proof}

 \begin{thm}\label{approx-thm-1}
Suppose that $A$ is symmetric and satisfies conditions (\ref{ellipticity}),
(\ref{periodicity}) and (\ref{smoothness}).
Let $u_\e\in H^1(\Omega)$ be the weak solution of the Dirichlet problem,
\begin{equation}\label{DP-1}
\left\{
\aligned
\mathcal{L}_\e (u_\e) & =0  & \quad  &\text{ in } \Omega,\\
u_\e & = f & \quad & \text{ on } \partial \Omega,
\endaligned
\right.
\end{equation}
where $f\in H^1(\partial \Omega)$ and $\Omega=B_r$ for some $9/8\le r\le 3/2$.
Then there exists $u_0\in H^1 (\Omega)$ such that
$\mathcal{L}_0 (u_0)=0$ in $\Omega$,
\begin{equation}\label{rate-10}
\|u_0\|_{C^1(B_1)} \le C\,  \| f\|_{L^2(\partial\Omega)},
\end{equation}
and
\begin{equation}\label{conv-rate}
\| u_\e  -u_0 \|_{L^\infty (B_{3/4})}
\le C \e\, \| f\|_{H^{1/2}(\partial \Omega)},
\end{equation}
where $C$ depends only on $d$, $\lambda$ and $M$.
 \end{thm}
 
 \begin{proof}
 Let $G_\e (x, y)$ and $G_0(x, y)$ denote the Green functions in $\Omega$ for $\mathcal{L}_\e$ and $\mathcal{L}_0$,
 respectively.
 Fix $x\in \Omega$, let
 $$
 w^x_\e (y)=G_\e (x, y) -G_0(x, y) -\e \chi_j (y/\e) \frac{\partial }{\partial y_j} \big\{ G_0 (x, y)\big\},
 $$
 and 
 $$
 v^x_\e (y)=w^x_\e (y) \eta(y-x),
 $$
  where $\eta$ is a function in $C^\infty(\R^d)$ such that $0\le \eta\le 1$,
 $$
 \eta\equiv 0 \quad \text{ in } B(0,1/32) \quad \text{ and } \quad \eta\equiv 1 \quad \text{ in } \Omega \setminus B(0, 1/16).
 $$
 Since $A$ is symmetric and $\mathcal{L}_\e (u_\e)=0$ in $\Omega$,
 by the Green identity,
 \begin{equation}\label{estimate-2.01}
   \int_{\partial\Omega} \frac{\partial v^x_\e}{\partial \nu_\e} \cdot u_\e \, d\sigma
 -\int_{\partial\Omega}
 v^x_\e \cdot \frac{\partial u_\e}{\partial\nu_\e}\, d\sigma
 =-\int_\Omega \mathcal{L}_\e (v^x_\e) \cdot u_\e\, dy
 \end{equation}
 for any $x\in \Omega$,
 where 
 $$
 \frac{\partial u_\e}{\partial \nu_\e}=n\cdot A(y/\e)\nabla u_\e
 $$
  denotes the conormal derivative of $u_\e$
 on $\partial\Omega$ associated with the operator $\mathcal{L}_\e$.
 We will show that for any $x\in B_{3/4}$,
 \begin{equation}\label{estimate-2.1}
 \Big| \int_\Omega  \mathcal{L}_\e (v^x_\e)
 \cdot u_\e \, dy \Big|
 \le C \e \, \| f\|_{H^{1/2}(\partial\Omega)},
 \end{equation}
 where $C$ depends only on $d$, $\lambda$ and $M$.
 Note that if $x\in B_{3/4}$,
 $$
 \aligned
 \Big| \int_{\partial\Omega} v^x_\e \cdot \frac{\partial u_\e}{\partial\nu_\e}\, d\sigma (y) \Big|
 &=\e\,  \Big| \int_{\partial\Omega}
 \chi_j (y/\e) \frac{\partial}{\partial y_j} \big\{ G_0 (x, y)\big\}\frac{\partial u_\e}{\partial\nu_\e}\, d\sigma (y) \Big|\\
 &\le C \e \int_{\partial\Omega} |\nabla u_\e|\, d\sigma\\
 &\le C \e \, \|\nabla u_\e\|_{L^2(\partial\Omega)}\\
 &\le C \e\,  \| f\|_{H^1(\partial\Omega)},
 \endaligned
 $$
 where we have used the estimate $|\nabla_y G_0 (x, y)|\le C |x-y|^{1-d}$
 for the first inequality and the Rellich estimate,
 $$
 \| \nabla u_\e\|_{L^2(\partial\Omega)} \le  C\,  \|\nabla_{tan} u_\e\|_{L^2(\partial\Omega)}
 $$
(see \cite{KS-1}) for the last.
This, together with (\ref{estimate-2.01}) and (\ref{estimate-2.1}), gives
\begin{equation}\label{estimate-2.2}
\Big|\int_{\partial\Omega} 
\frac{\partial v^x_\e}{\partial \nu_\e} \cdot u_\e\, d\sigma \Big|
\le C \e\,  \| f\|_{H^1(\partial\Omega)}
\end{equation}
for any $x\in B_{3/4}$.
Observe that on $\partial\Omega$, $u_\e=f$ and that if $x\in B_{3/4}$,
\begin{equation}\label{estimate-2.3}
\aligned
\frac{\partial v^x_\e}{\partial \nu_\e}
&=\frac{\partial w^x_\e}{\partial \nu_\e}\\
&=\frac{\partial }{\partial \nu_\e (y)} \Big\{ G_\e (x, y) \Big\}
-\frac{\partial}{\partial \nu_\e (y)} \Big\{ G_0(x, y) \Big\}\\
&\qquad\qquad
-\frac{\partial}{\partial \nu_\e (y)}
\Big\{ \e \chi_j (y/\e)\Big\} \frac{\partial }{\partial y_j}  \big\{ G_0 (x, y) \big\}\\
&\qquad\qquad
-\e \chi_j (y/\e) \frac{\partial}{\partial \nu_\e(y)}
\left\{ \frac{\partial}{\partial y_j} \big\{ G_0 (x, y) \big\} \right\}.
\endaligned
\end{equation}
We now let
\begin{equation}\label{estimate-2.4}
\aligned
u_0(x)
 &=-\int_{\partial\Omega} \frac{\partial}{\partial \nu_\e (y)} 
 \Big\{ G_0 (x, y) \Big\} f(y)\, d\sigma(y)\\
 &\qquad
 -\int_{\partial\Omega}
 \frac{\partial}{\partial \nu_\e (y)}
\Big\{ \e \chi_j (y/\e)\Big\} \frac{\partial }{\partial y_j}  \big\{ G_0 (x, y) \big\} f(y)\, d\sigma (y)
\endaligned
\end{equation}
for $x\in \Omega$.
Then $\mathcal{L}_0 (u_0)=0$ in $\Omega$, and
$$
\| u_0\|_{C^1(B_1)}
\le C \int_{\partial\Omega} | f|\, d\sigma
\le C\,  \| f\|_{L^2(\partial\Omega)},
$$
where we have used the estimate $\|\nabla \chi\|_\infty\le C$.
Since
$$
u_\e (x)=-\int_{\partial\Omega} 
\frac{\partial }{\partial \nu_\e (y)} \Big\{ G_\e (x, y) \Big\} f(y)\, d\sigma (y),
$$
it follows from (\ref{estimate-2.2}) and (\ref{estimate-2.3}) that for any $x\in B_{3/4}$,
\begin{equation}\label{estimate-2.6}
| u_\e (x)-u_0 (x)|\le 
C \e \, \| f\|_{H^1(\partial\Omega)}, 
\end{equation}
where we have used the estimate $|\nabla_y^2 G_0(x, y)|\le C |x-y|^{-d}$.

It remains to prove (\ref{estimate-2.1}).
To this end, we first note that
$$
\aligned
\mathcal{L}_\e (v^x_\e)
&=-\text{\rm div} \big(A^\e \nabla (w_\e^x \eta^x)\big)\\
&= -\text{\rm div} \big(A^\e \nabla w^x_\e\big) \eta^x
-A^\e\nabla w^x_\e \cdot \nabla \eta^x -\text{\rm div} \big(A^\e (\nabla \eta^x)w_\e^x\big),
\endaligned
$$
where  $A^\e (y)=A(y/\e)$ and $\eta^x (y)=\eta (y-x)$.
Since 
$$
\mathcal{L}_\e \big\{ G_\e (x, \cdot)\big\} =\mathcal{L}_0 \big\{ G_0(x, \cdot)\big\}=0 \quad \text{ in } 
\Omega \setminus \{ x\},
$$
it follows by Lemma \ref{lemma-2.1} that
\begin{equation}\label{w-eq}
\aligned
-\text{\rm div} \big( A^\e \nabla w_\e^x \big)
&=-\e \frac{\partial}{\partial y_i}
\left\{ \phi_{kij} (y/\e) \frac{\partial^2}{\partial y_k\partial y_j}
\big\{ G_0 (x, y) \big\} \right\}\\
&\qquad
+\e\,  \text{\rm div}\left\{ A(y/\e) \chi_j(y/\e) \nabla _y \frac{\partial}{\partial y_j} \big\{ G_0 (x, y) \big\} \right\},
\endaligned
\end{equation}
where $\phi=(\phi_{kij})$ is given by Lemma \ref{flux-lemma}.
Using integration by parts,
this leads to
\begin{equation}\label{estimate-2.10}
\aligned
\Big| \int_\Omega
\mathcal{L}_\e (v_\e^x) \cdot u_\e\, dy \Big|
&\le C \e\int_{\partial\Omega} |\nabla_y^2 G_0(x, y)| |u_\e|\, d\sigma (y) \\
&\qquad+ C \e \int_\Omega 
|\nabla_y^2 G_0 (x, y)| |\nabla \eta^x| |u_\e|\, dy\\
&\qquad + C \e \int_\Omega  |\nabla^2_y G_0 (x, y)| |\eta^x| |\nabla u_\e|\, dy\\
& \qquad + C \int_\Omega |\nabla w_\e^x ||\nabla \eta^x| |u_\e|\, dy\\
&\qquad+ C \int_\Omega |w_\e^x| |\nabla \eta^x |\nabla u_\e|\, dy,
\endaligned
\end{equation}
where we have used the fact $\|\chi\|_\infty +\|\phi\|_\infty \le C$.

Finally, to bound the RHS of (\ref{estimate-2.10}), we use the fact that 
$|\nabla^2_y G_0 (x, y)|\le C |x-y|^{-d}$ and $\eta^x = 0$ in $B(x, 1/32)$.
As a result, the first three terms in the RHS of (\ref{estimate-2.10}) are bounded by
$$
C\e \int_{\partial\Omega} |u_\e|\, d\sigma
+ C \e \int_\Omega \big( |\nabla u_\e| +|u_\e|\big)\, dy.
$$
Note that for $y\in \Omega\setminus B(x, 1/64)$,
$$
\aligned
|w^x_\e (y)|
& \le |G_\e (x, y)-G_0(x, y)| + C \e |\nabla_y G_0(x, y)|\\
&\le C \e |x-y|^{1-d}\\
& \le C\e
\endaligned
$$
 (see \cite{KLS-2014}).
 Also, in view of (\ref{w-eq}), we may use  Caccioppoli's inequality to deduce that
 $$
 \aligned
&  \int_{B(x, 1/16)\setminus B(x, 1/32)}
 |\nabla w_\e^x|^2\, dy\\
 & \quad \le  C \int_{B(x, 1/8)\setminus B(x, 1/64)}
 \big( | w_\e^x|^2 + \e^2 |\nabla_y^2 G_0(x, y)|^2 \big) \, dy \\
 &\quad \le C \e^2.
 \endaligned
 $$
 Hence, the last two terms in the RHS of (\ref{estimate-2.10})
 are bounded by
 $$
 C\e \left(\int_\Omega \big( |\nabla u_\e|^2 +|u_\e|^2 \big)\, dy \right)^{1/2}.
 $$
 In summary, we have proved that for any $x\in B_{3/4}$,
 $$
 \aligned
 \Big| \int_\Omega
\mathcal{L}_\e (v_\e^x) \cdot u_\e\, dy \Big|
& \le C \e \big\{  \| u_\e\|_{H^1(\Omega)} +\| u_\e\|_{L^1(\partial\Omega)} \big\} \\
&\le C \e \, \| f\|_{H^{1/2}(\partial\Omega)}.
\endaligned
$$
 This completes the proof.
 \end{proof}
 
 \begin{remark}
 {\rm
 Theorem \ref{approx-thm-1} continues to hold for a bounded $C^{2, \alpha}$
 domain $\Omega$ in $\R^d$ with $B_{3/4}$ replaced by any subdomain 
 $\Omega^\prime$ such that dist$(\Omega^\prime, \partial\Omega)>0$. 
 The smoothness condition on $\partial\Omega$ ensures the 
 pointwise estimate $|\nabla_y^2 G_0(x, y)|\le C |x-y|^{-d}$ for $y\in \overline{\Omega}$.
 }
 \end{remark}
 
 \begin{remark}
 {\rm
 The function $u_0$ given by (\ref{estimate-2.4}) does not agree with $u_\e$ on $\partial\Omega$.
 Indeed, using the fact that $G_0(x, y)=0$ for $y\in \partial\Omega$ and thus
 $$
 \frac{\partial}{\partial y_j} \big\{ G_0 (x, y) \big\}
 = n_j (y) \frac{\partial}{\partial n(y)} \big\{ G_0(x, y)\big\},
 $$
 it is not hard to see that $u_0 =\omega_\e f$ on $\partial\Omega$,
 where $w_\e(y)=h(y, y/\e)$  and $h(x, y)$ is 1-periodic in the $y$ variable.
 In particular, we have
 $$
 \|\omega_\e \|_{L^\infty(\partial\Omega)} \le C \quad \text{ and } \quad 
 \|\omega_\e\|_{H^{1/2}(\partial\Omega)} \le C \e^{-1/2}.
 $$
 By the square function estimate for $\mathcal{L}_0$, we obtain 
 \begin{equation}\label{sq-estimate}
 \aligned
&  \int_\Omega |\nabla u_0(x) |^2 \, \text{\rm dist}(x, \partial\Omega)\, dx
 +\int_\Omega |u_0(x)|^2\, dx\\
 & \qquad
 \le C \, \| u_0\|_{L^2(\partial\Omega)}^2
 \le  C\,  \| f\|_{L^2(\partial\Omega)}^2,
 \endaligned
 \end{equation}
 where $C$ depends only on $d$, $\lambda$ and $M$.
 Estimate (\ref{sq-estimate}) is not used in this paper.
 }
 \end{remark}
 
 We now give the proof of Theorem \ref{main-theorem-3}.
  
 \begin{proof}[\bf Proof of Theorem \ref{main-theorem-3}]
 
 By Caccioppoli's inequality,
 \begin{equation}\label{2-31}
 \int_{B_{5/4}} |\nabla u_\e|^2 \, dx
 \le C \int_{B_{3/2}} |u_\e|^2\, dx.
 \end{equation}
 It follows that there exists some $r\in (9/8,5/4)$ such that
 \begin{equation}\label{2-32}
 \int_{\partial B_r} |\nabla u_\e|^2\, d\sigma  +\int_{\partial B_r} |u_\e|^2\, d\sigma
 \le C \int_{B_{3/2}} |u_\e|^2\, dx.
 \end{equation}
 For otherwise we may integrate the reverse inequality of (\ref{2-32}) in $r$ over $(9/8, 5/4)$
 to obtain an inequality that is in contradiction with (\ref{2-31}).
 We now apply Theorem \ref{approx-thm-1} to $u_\e$ in $\Omega=B_r$.
 This gives us a function $u_0\in H^1(B_r)$ such that
 $\mathcal{L}_0 (u_0)=0$ in $B_r$ and
 $$
 \aligned
 \| u_\e -u_0\|_{L^\infty(B_{3/4})}
  &\le C \e \, \| u_\e\|_{H^1(\partial B_r)}\\
 &\le C \e\, \| u_\e\|_{L^2(B_{3/2})},
 \endaligned
 $$
 where we have used (\ref{2-32}) for the last step.
 We also obtain from (\ref{estimate-2.4}) that 
 \begin{equation}\label{app-300}
 \aligned
 \| u_0\|_{C^1(B_1)} 
 \le C \, \| u_\e\|_{L^2(\partial B_r)}
 \le C \, \| u_\e\|_{L^2(B_{3/2})},
 \endaligned
 \end{equation}
 where $C$ depends only on $d$, $\lambda$ and $M$.
 
 Suppose now that
 \begin{equation}\label{app-100}
 \fint_{B_{2}} |u_\e|^2 \le N \fint_{B_{\sqrt{\lambda}}} |u_\e|^2\, dx
 \end{equation}
 for some $N>1$.
It follows from (\ref{app-300}) that
\begin{equation}\label{app-101}
\left(\fint_{B_1} |u_0|^2\,dx\right)^{1/2}  \le C \left(\fint_{B_{3/2}} |u_\e|^2\, dx\right)^{1/2},
\end{equation}
and
$$
\aligned
\left(\fint_{B_{1/2}} |u_0|^2\, dx\right)^{1/2}
 & \ge \left(\fint_{B_{1/2}} |u_\e|^2\, dx\right)^{1/2}-\left(\fint_{B_{1/2}} |u_\e- u_0|^2\, dx\right)^{1/2}\\
 &\ge \left(\fint_{B_{1/2}} |u_\e|^2\, dx\right)^{1/2}
 -C \e \left(\fint_{B_{3/2}} |u_\e|^2\, dx\right)^{1/2}\\
 &\ge \left([C(N)]^{-1}-C \e \right)
  \left(\fint_{B_{3/2}} |u_\e|^2\, dx\right)^{1/2},
 \endaligned
 $$
 where the last step follows from Theorem \ref{main-theorem-2}.
 We should point out that the proof of Theorem \ref{main-theorem-2} in the next section
 does not use Theorem \ref{main-theorem-3}.
 Thus,  in view of (\ref{app-101}),  if $ C (N) \e <1/2$, the solution $u_0$ satisfies
 \begin{equation}\label{app-102}
 \fint_{B_1} |u_0|^2\, dx 
 \le C (N) \fint_{B_{1/2}} |u_0|^2\, dx,
 \end{equation}
 where $C(N)$ depends only on $d$, $\lambda$, $M$ and $N$.
 \end{proof}

 
\section{\bf Uniform doubling conditions}

Fix $\lambda\in (0, 1]$ and $M>0$.
Let $\mathcal{A}=\mathcal{A}(\lambda, M)$ denote the set of all $d\times d$ symmetric matrices $A=A(y)$ that satisfy the 
conditions (\ref{ellipticity}), (\ref{periodicity}) and (\ref{smoothness}).
For each $A\in \mathcal{A}$, we introduce a family of ellipsoids, 
\begin{equation}\label{E}
E_r(A)=\big\{ x\in \R^d: \ \langle (\widehat{A})^{-1}x,   x\rangle <r^2 \big\},
\end{equation}
where $\widehat{A}$ is the homogenized matrix defined by (\ref{gomo-c}).
Since $\widehat{A}$ satisfies (\ref{ellipticity}), we have
\begin{equation}\label{relation}
B(0, r{\sqrt{\lambda}}) \subset E_r (A)\subset B (0, r)
\end{equation}
for $0<r<\infty$.

The goal of this section is to prove the following.

\begin{thm}\label{d-thm}
Let $u_\e\in H^1(B_2)$ be a weak solution of $\text{\rm div}\big(A(x/\e)\nabla u_\e\big)=0$ in $B_2$ for some
$A\in \mathcal{A}$.
Suppose that 
\begin{equation}\label{dd}
\fint_{E_2(A)} | u_\e|^2\, dx \le N \fint_{E_1(A)} |u_\e|^2\, dx
\end{equation}
for some $N>1$.
Then for $0<r\le 1$,
\begin{equation}\label{d-condition}
\fint_{E_{r}(A)} |u_\e|^2\, dx \le C(N) \fint_{E_{r/2}(A)} |u_\e|^2\, dx,
\end{equation}
where $C(N)$ depends only on $d$, $\lambda$, $M$ and $N$.
\end{thm}

To prove Theorem \ref{d-thm}, we first note that
if $\e\ge \e_0>0$, then 
$$
|A(x/\e)-A(y/\e)|\le M\e^{-1} |x-y|\le M\e_0^{-1} |x-y|
$$
for any $x, y\in \R^d$.
As a result, the estimate (\ref{d-condition}) follows directly from \cite{G-Lin}.
In this case the periodicity of $A$ is not needed and the constant $C(N)$ in (\ref{d-condition}) depends on $\e_0$.
One may  also replace $E_r(A)$ by the ball $B_r$.

The proof for the case $0<\e<\e_0$ uses a compactness argument from the homogenization theory.

\begin{lemma}\label{d-lemma}
Let $m\ge 1$ be a positive integer.
Then there exists $\e_0>0$, depending only on $d$, $m$ and $\lambda$, such that
\begin{equation}\label{d-1}
\fint_{E_1(A)} |u_\e|^2\, dx \le 2^{2m+1} \fint_{E_{1/2}(A)} |u_\e|^2\, dx,
\end{equation}
whenever $0<\e\le \e_0$, $u_\e\in H^1(B_2)$ is a weak solution of
$\text{\rm div}(A(x/\e)\nabla u_\e)=0$ in $B_2$ for some symmetric matrix
$A$ satisfying (\ref{ellipticity}) and (\ref{periodicity}), and
\begin{equation}\label{d-2}
\fint_{E_{2}(A)} |u_\e|^2\, dx \le 2^{2m+1} \fint_{E_{1}(A)} |u_\e|^2\, dx.
\end{equation}
\end{lemma}

\begin{proof}
We prove the lemma by contradiction.
Suppose that there exist sequences $\{ \e_k\}\subset \R_+$, $\{ A_k \}$ satisfying (\ref{ellipticity}) and (\ref{periodicity}),
$\{ u_{k} \}\subset H^1(B_2)$, such that $\e_k \to 0$,
\begin{equation}\label{d-3}
\text{\rm div} \big( A_k (x/\e_k)\nabla u_{k}\big) =0 \quad \text{ in } B_2,
\end{equation}
\begin{equation}\label{d-4}
\fint_{E_{1}(A_k)} |u_k|^2\, dx > 2^{2m+1} \fint_{E_{1/2}(A_k)} |u_k|^2\, dx,
\end{equation}
and
\begin{equation}\label{d-5}
\fint_{E_{2}(A_k)} |u_k|^2\, dx \le 2^{2m+1} \fint_{E_{1}(A_k)} |u_k|^2\, dx,
\end{equation}
where $E_r(A_k)$ is defined by (\ref{E}).
Since $\big\{\widehat{A_k}\big\}$ is symmetric and bounded in $\R^{d\times d}$, we may assume that 
\begin{equation}\label{d-7}
\widehat{A_k}\to H
\end{equation}
for some symmetric matrix $H$ satisfying (\ref{ellipticity}).
By multiplying a constant to $u_k$, we may assume that
\begin{equation}\label{norm1}
\fint_{E_{2}(A_k)} |u_k|^2\, dx=1.
\end{equation}
By Caccioppoli's inequality this implies that $\{u_k\}$ is bounded in $H^1(E_r(H))$ for any $0<r<2$.
Thus, by passing to a subsequence, we may further assume that
\begin{equation}\label{d-6}
\aligned
u_k   & \to u & \quad & \text{ weakly in } H^1(E_r(H)),\\
A_k(x/\e_k)\nabla u_k  & \to F &\quad &\text{ weakly in } L^2(E_r(H))
\endaligned 
\end{equation}
for any $0<r<2$, where $u\in H_{\text{loc}}^1(E_2(H))$ and $F\in L_{\text{loc}}^2(E_2(H))$.
It follows from the theory of homogenization (see e.g. \cite{JKO-1994}) that $F=H\nabla u$ and 
\begin{equation}\label{d-8}
\text{\rm div} \big(H \nabla u) =0 \quad \text{ in } E_2(H).
\end{equation}

To proceed, we note that the weak convergence of $u_k$ in $H^1(E_r(H))$ implies 
$u_k \to u$ strongly in $L^2(E_r(H))$. In view of (\ref{d-4}), (\ref{d-5}), (\ref{d-7}) and (\ref{norm1}), 
by letting $k\to \infty$, we may deduce that
\begin{equation}\label{d-9}
\fint_{E_{1}(H)} |u|^2\, dx \ge 2^{2m+1} \fint_{E_{1/2}(H)} |u|^2\, dx,
\end{equation}
and
\begin{equation}\label{d-10}
\fint_{E_{2}(H)} |u|^2\, dx\le 1 \le 2^{2m+1} \fint_{E_{1}(H)} |u|^2\, dx.
\end{equation}
Since $H$  is symmetric and positive definite,
there exists a $d\times d$ matrix $S$ such that $SHS^T=I_{d\times d}$.
Let $u(x)=w(Sx)$.
Then
$$
\Delta w=\text{\rm div}\big(SHS^T\nabla w)=\text{\rm div} \big(H\nabla u\big)=0.
$$
Note that $H^{-1}=S^T S$ and 
$$
\langle H^{-1} x, x\rangle =|Sx|^2.
$$
By a change of variables it follows from (\ref{d-9}) and (\ref{d-10}) that
\begin{equation}\label{d-11}
\fint_{B_{1}} |w|^2\, dx \ge 2^{2m+1} \fint_{B_{1/2}} |w|^2\, dx,
\end{equation}
and
\begin{equation}\label{d-12}
\fint_{B_{2}} |w|^2\, dx \le 1 \le 2^{2m+1} \fint_{B_{1}} |w|^2\, dx.
\end{equation}

Next, we use the fact that for the harmonic function $w$ in $B_2$, 
the function 
$$
\psi (r)=\log_2 \left(\fint_{B_{2^r}} |w|^2\, dx \right)
$$
is a convex function of $r$ on the interval $(-\infty, 1]$ 
(this is a consequence of the well-known three-spheres theorem for harmonic functions).
It follows that
$$
\aligned
\fint_{B_{1}} |w|^2\, dx 
 & \le \left(\fint_{B_{1/2}} |w|^2\, dx \right)^{1/2}
 \left(\fint_{B_{2}} |w|^2\, dx \right)^{1/2}\\
 & \le 2^{m+\frac12} \left(\fint_{B_{1/2}} |w|^2\, dx \right)^{1/2}
 \left(\fint_{B_{1}} |w|^2\, dx \right)^{1/2},
 \endaligned
$$
where we have used (\ref{d-12}) for the last step.
This, together with (\ref{d-11}), yields 
\begin{equation}\label{d-13}
\fint_{B_{1}} |w|^2\, dx = 2^{2m+1} \fint_{B_{1/2}} |w|^2\, dx.
\end{equation}
Using (\ref{d-12}) and (\ref{d-13}), we obtain 
$$
\psi(0) =(1/2) \psi(-1) +(1/2)\psi(1).
$$
By the convexity of $\psi$, it follows that $\psi$ is a linear function on the
interval $[-1,1]$.
Since $\psi$ is analytic on $(-\infty, 1)$,
we may conclude that $\psi$ is a linear function on $(-\infty, 1]$.
It follows that 
\begin{equation}\label{d-14}
\frac{\fint_{B_{2r}} |w|^2\, dx}{\fint_{B_r} |w|^2\, dx }
=\frac{\fint_{B_{1/2}} |w|^2\, dx}{\fint_{B_{1/4}} |w|^2\, dx }
=2^{2m+1}
\end{equation}
for any $0<r\le 1$.

Finally, we write 
$$
w(x)=P_\ell (x) +R_\ell (x),
$$
where $P_\ell  (x)$ is a homogeneous polynomial of degree $\ell\ge 0$ and
the remainder $R_\ell (x)$ satisfies the estimate
$$
|R_k (x)|\le C_w |x|^{\ell+1} \quad \text{ for }  x\in B_{1}.
$$
It is not hard to see that as $ r\to 0$,
\begin{equation}\label{d-15}
\fint_{B_r} |w|^2 dx= r^{2\ell}\fint_{B_1} |P_\ell|^2\, dx
+ O(r^{2\ell +1}).
\end{equation}
This, together with (\ref{d-14}), implies that $2\ell=2m+1$,
which is in contradiction with the assumption that $m$ is an integer.
\end{proof}

\begin{proof}[\bf Proof of Theorem \ref{d-thm}]

Let $u_\e\in H^1(B_2)$ be a solution of
$\text{div}\big(A(x/\e)\nabla u_\e\big)=0$ in $B_2$ for some $A\in \mathcal{A}$.
Suppose that
\begin{equation}\label{d-40}
\fint_{E_2(A)} |u_\e|^2\, dx \le N \fint_{E_{1}(A)} |u_\e|^2\, dx,
\end{equation}
where $N>2$.
Let $m$ be an integer such that $2^{2m+1}\ge N\ge 2^{2m-1}$.
Let $\e_0>0$, which depends on $d$, $\lambda$ and $m$, be given by Lemma \ref{d-lemma}.
We may assume that $0<\e< \e_0$.
For otherwise the inequality (\ref{d-condition}) follows from \cite{G-Lin}, as we pointed out earlier.

It follows from (\ref{d-40}) by Lemma \ref{d-lemma} that
\begin{equation}\label{d-41}
\fint_{E_{1}(A)} |u_\e|^2\, dx \le 2^{2m+1} \fint_{E_{1/2}(A)} |u_\e|^2\, dx.
\end{equation}
Let $v(x)=u_\e (x/2)$.
Note that 
$
\mathcal{L}_{2\e} (v)=0  \text{ in } B_2,
$
By (\ref{d-41}) and a change of variables,
$$
\fint_{E_{2}(A)} |v|^2\, dx \le 2^{2m+1} \fint_{E_{1}(A)} |v|^2\, dx.
$$
Thus, if $2\e\le \e_0$, we may use Lemma \ref{d-lemma} again to obtain 
$$
\fint_{E_{1}(A)} |v|^2\, dx \le 2^{2m+1} \fint_{E_{1/2}(A)} |v|^2\, dx,
$$
which, by a change of variables, leads to 
$$
\fint_{E_{1/2}(A)} |u_\e|^2\, dx \le 2^{2m+1} \fint_{E_{1/4}(A)} |u_\e|^2\, dx.
$$
By an induction argument  we see that if $2^{k-1} \e\le \e_0$, 
\begin{equation}\label{d-45}
\fint_{E_{2^{-k+1}}(A)} |u_\e|^2\, dx \le 2^{2m+1} \fint_{E_{2^{-k}}(A)} |u_\e|^2\, dx.
\end{equation}
Suppose now that $(\e/\e_0)\le r\le 1$.
Let $k$ be an integer such that $2^{-k} \le r\le 2^{-k+1}$.
Then $(\e/\e_0)\le 2^{-k+1}$.
It follows from (\ref{d-45}) that
\begin{equation}\label{d-49}
\aligned
\fint_{E_r (A)} |u_\e|^2\, dx
&\le C \fint_{E_{2^{-k+1}}(A)} |u_\e|^2\, dx
\le C 2^{2m} \fint_{E_{2^{-k}}(A)} |u_\e|^2\, dx\\
& \le C2^{4m} \fint_{E_{2^{-k-1}}(A)} |u_\e|^2\, dx
\le C 2^{4m} \fint_{E_{r/2}(A)} |u_\e|^2\, dx,
\endaligned
\end{equation}
where $C$ depends only on $d$ and $\lambda$.

Finally, to deal with the case $0<r<(\e/\e_0)$, we use a blow-up argument.
Let $w(x)=u_\e (\e x/\e_0)$. Then
$
\mathcal{L}_{\e_0} (w)=0.
$
Note that by (\ref{d-49}) with $r=\e/\e_0$,
$$
\fint_{E_1(A)} |w|^2\, dx \le  C 2^{4m} \fint_{E_{1/2}(A)} |w|^2\, dx.
$$
It follows from \cite{G-Lin}  that for $0<r<1$,
$$
\fint_{E_r(A)} |w|^2\, dx \le {C} \fint_{E_{r/2}(A)} |w|^2\, dx,
$$
where ${C}$ depends only on $d$, $\lambda$, $M/\e_0$ and $m$.
By a change of variables this yields
\begin{equation}\label{d-50}
\fint_{E_r (A)} |u_\e|^2\, dx \le {C} \fint_{E_{r/2}(A)} |u_\e|^2\, dx,
\end{equation}
for any $0<r<\e/\e_0$.
In view of (\ref{d-49}) and (\ref{d-50}) we have proved that the inequality (\ref{d-50}) holds for any $0<r\le 1$,
where $C$ depends only on $d$, $\lambda$, $M$ and $N$.
\end{proof}

We now deduce Theorem  \ref{main-theorem-2} from Theorem \ref{d-thm}, using (\ref{relation}).

\begin{proof}[\bf Proof of Theorem \ref{main-theorem-2}]
By (\ref{relation}) the condition (\ref{double-1}) implies 
$$
\int_{E_2(A)} |u_\e|^2\, dx \le C N \fint_{E_1(A)} |u_\e|^2\, dx.
$$
It follows by Theorem \ref{d-thm} that if $0<r\le \sqrt{\lambda}$,
$$
\aligned
\fint_{B_r} |u_\e|^2\, dx 
&\le C \fint_{E_{r/\sqrt{\lambda}} (A)} |u_\e|^2\, dx
\le C(N) \fint_{E_{2^{-\ell }r/\sqrt{\lambda}}} |u_\e|^2\, dx\\
&\le C (N) \fint_{B_{r/2}} |u_\e|^2\, dx,
\endaligned
$$
where $\ell \ge 1$ is an integer such that $2^{-\ell}\le \sqrt{\lambda}/2$.
\end{proof}

The next theorem will be used in the proof of Theorem \ref{main-theorem-1}.

\begin{thm}\label{d-thm-1}
Let $u_\e$ be a weak solution of
$\text{\rm div}\big(A(x/\e)\nabla u_\e\big)=0$ in $B_2$ for some
$A\in \mathcal{A}$.
Suppose that
\begin{equation}\label{d-200}
\fint_{B_{2}} |u_\e|^2\, dx \le N \fint_{B_{\sqrt{\lambda}}} |u_\e|^2\, dx
\end{equation}
for some $N>1$.
Then  for  any $0<r<3/4$ and $|x_0|\le \sqrt{\lambda}/2$,
\begin{equation}\label{d-201}
\fint_{B(x_0, r)} |u_\e|^2\, dx \le C(N) \fint_{B(x_0, r/2)} |u_\e|^2\, dx,
\end{equation}
where $C(N)$ depends only on $d$, $\lambda$, $M$ and $N$.
\end{thm}

\begin{proof}
The case $x_0=0$ is contained in Theorem \ref{main-theorem-2}.
To handle the general case where $|x_0|\le \sqrt{\lambda}/4$, we use the fact that
$\mathcal{A}$ is translation invariant.
Let
$
v(x)=u_\e (x_0 +tx).
$
Then 
$$
\text{\rm div} \big(\widetilde{A}(x/(\e t^{-1}))\nabla v\big)=0 \quad \text{ in } B_2,
$$
where  $\widetilde{A}(y)= A(y + x_0/\e)$.
Observe that $\widetilde{A}\in \mathcal{A}$ and that if $t=3/4$ and $|x_0|\le \sqrt{\lambda}/2$,
$B(0, c_0)\subset B(x_0, t \sqrt{\lambda})$ for $c_0=\sqrt{\lambda}/8$. Hence,
$$
\aligned
\fint_{B(0,2)} |v|^2\, dx
&=\fint_{B(x_0, 2t)} |u_\e|^2\, dx
\le C \fint_{B(0,2)} |u_\e|^2\, dx\\
&\le C(N)   \fint_{B(0, c_0) } |u_\e|^2\, dx
\le C(N)\fint_{B(x_0, t\sqrt{\lambda})} |u_\e|^2\, dx\\
&  = C(N) \fint_{B(0, \sqrt{\lambda})} |v|^2\, dx,
\endaligned
$$
where we have used Theorem \ref{main-theorem-2}.
It follows by Theoem \ref{main-theorem-2} that for $0<r<2$,
$$
\fint_{B(0, r)} |v|^2\, dx \le C \fint_{B(0, r/2)} |v|^2\, dx,
$$
which gives (\ref{d-201}).
\end{proof}


\section{\bf Uniform measure estimates of nodal sets}

Throughout this section 
$u_\e$ is a nonzero weak solution of $\text{\rm div}\big(A(x/\e)\nabla u_\e)=0$ in
$B_{2}$ for some $A\in \mathcal{A}(\lambda, M)$.
In view of Theorem \ref{d-thm-1} we assume that $u_\e$ satisfies the doubling condition,
\begin{equation}\label{ud-100}
\fint_{B(y, r)} |u_\e|^2\, dx
\le \widetilde{N} \fint_{B(y, r/2)} |u_\e|^2\, dx
\end{equation}
for any $|y|\le \sqrt{\lambda}/2$ and $0<r<3/4$, where $\widetilde{N}>1$.
Without the loss of generality we further assume that
\begin{equation}\label{m-0}
\int_{B_{2}} |u_\e|^2\, dx =1.
\end{equation}
Let
\begin{equation}\label{z-set}
Z(u_\e) =\big\{ x\in B_2: \ u_\e (x) =0\big\}
\end{equation}
denote the nodal set of $u_\e$.
Define
\begin{equation}\label{F}
F_\e (y, r)=\frac{\mathcal{H}^{d-1}\big(Z(u_\e)\cap B(y, r)\big)}{ r^{d-1}},
\end{equation}
where $B(y, r)\subset B(0,2)$.
Our goal is to prove that
\begin{equation}\label{nd-1}
F_\e (0, \sqrt{\lambda}/4) \le C(\widetilde{N}),
\end{equation}
where $C(\widetilde{N})$ depends only on $d$, $\lambda$, $M$ and $\widetilde{N}$.

We first recall one of the main results in \cite{Han-Lin-2}.

\begin{thm}\label{cd-thm}
Let $u$ be a nonzero weak solution of $\text{\rm div}\big( A(x)\nabla u)=0$ in $B_1$
for some symmetric matrix $A$ satisfying  conditions (\ref{ellipticity}) and (\ref{smoothness}).
Suppose that
\begin{equation}\label{c-d}
\fint_{B_1} |u|^2\, dx \le {N} \fint_{B_{1/2}} |u|^2\, dx
\end{equation}
for some ${N}>1$. Then 
\begin{equation}\label{cd-1}
\mathcal{H}^{d-1} \big( Z(u)\cap B_{1/2}\big)
\le C({N}),
\end{equation}
where $C({N})$ depends only on $d$, $\lambda$, $M$ and $\widetilde{N}$.
\end{thm}

If $\e\ge \e_0>0$, the estimate (\ref{nd-1}) follows directly from Theorem \ref{cd-thm}, with constant $C(\widetilde{N})$ 
also depending on $\e_0$.
We may also use Theorem \ref{cd-thm}  to obtain 
a small-scale estimate by a rescaling argument.
The estimate allows us to bound the function $F_\e (y, r)$ for $0<r<C\e$.

\begin{lemma}\label{s-lemma}
Suppose that $0<r<k\e $ for some $k\ge 1$ and $r<(3/8)$. Then for any  $|x_0|\le \sqrt{\lambda}/2$,
\begin{equation}\label{u-1}
\mathcal{H}^{d-1} \big( Z(u_\e)\cap B(x_0, r) \big)
\le C(\widetilde{N}, k) r^{d-1},
\end{equation}
where $C(\widetilde{N}, k)$ depends only on $d$, $\lambda$, $k$, $M$ and $\widetilde{N}$.
\end{lemma}

\begin{proof}
Let $v(x)=u_\e (x_0 + 2rx)$. Then
$$
\text{\rm div}\big( \widetilde{A}(x/(\e/2r))\nabla v\big) =0 \quad \text{ in } B_1,
$$
where $\widetilde{A}(y)=A(y+x_0/\e)$.
Since $\e/(2r)\ge 1/(2k)$ and
$$
\aligned
\fint_{B_1} |v|^2\, dx
 & =\fint_{B(x_0, 2r)} |u_\e|^2\, dx\\
 & \le\wN \fint_{B(x_0, r)} |u_\e|^2\, dx 
 =\widetilde{N} \fint_{B_{1/2}} |v|^2\, dx,
 \endaligned
 $$
 where we have used (\ref{ud-100}), it follows by Theorem \ref{cd-thm}  that
$$
\mathcal{H}^{d-1}
\big( Z(v)\cap B_{1/2}\big)
\le C(\widetilde{N}, k),
$$
where $C(\widetilde{N}, k)$ depends only on $d$, $\lambda$,  $k$, $M$ and $\widetilde{N}$.
By a change of variables this yields (\ref{u-1}).
\end{proof}

For $N>1$, define
\begin{equation}
\aligned
\mathcal{F} (N)
=\bigg\{ u\in H^1(B_1): \  & \text{\rm div}\big(A\nabla u\big)=0 \text{ in } B_1 \text{ for some constant matrix } \\
& \ A\in \mathcal{A}(\lambda, M)
\text{ and } 
1=\fint_{B_1} |u|^2\, dx \le N \fint_{1/2} |u|^2\, dx \bigg\}.
\endaligned
\end{equation}
Let 
\begin{equation}\label{s-set}
S(u)=\big\{ x\in B_1: \ u(x)=|\nabla u(x)|=0 \big\}
\end{equation}
denote the singular set of $u$.
It was proved in \cite{Han-Lin-H-1998} that if $u\in \mathcal{F}(N)$,
\begin{equation}\label{s-estimate}
\mathcal{H}^{d-2} \big( S(u)\cap B(0, r) \big)\le C(N) r^{d-2}
\end{equation}
for any $0<r\le 1/2$, where $C(N)$ depends only on $d$, $\lambda$ and $N$.

\begin{lemma}\label{c-lemma}
Let $r_0=\sqrt{\lambda}/4$.
Then there exist $r_1 \in (0, r_0/4)$ and $\delta_1>0$, depending only on $d$, $\lambda$ and $N$, 
such that for each $u\in \mathcal{F}(N)$, there  exist
 two finite sequences of balls 
$$
\big\{B(x_i, t_i ), i=1,2, \dots, m \big\} \quad
\text{ and  } \quad
\big\{ B(y_j, s_j), j=1,2, \dots, m^\prime \big\},
$$
with the properties  that $x_i, y_j \in B(0, r_0)$, $0<t_i, s_j <r_1$,
\begin{equation}\label{c-1}
\Big\{ x\in B(0, r_0): \ |u(x)|< \delta_1 \Big\}
\subset
\bigcup_{i=1}^m B(x_i, t_i)
\end{equation}
\begin{equation}\label{c-1.5}
\Big\{ x\in B(0, r_0): \ |u(x)|+|\nabla u(x)|<\delta_1 \Big\} \subset
 \bigcup_{j=1}^{m^\prime} B(y_j, s_j),
\end{equation}
\begin{equation}\label{c-2}
 \sum_{i=1}^m t_i ^{d-1}  \le C(N) \quad \text{ and } \quad
\sum_{j=1}^{m^\prime} s_j^{d-1}  <\frac14  r_0^{d-1},
\end{equation}
where  $C(N)$ depends only on $d$, $\lambda$ and $N$.
\end{lemma}

\begin{proof}
It follows from (\ref{s-estimate}) that  for $0<r<r_0/4$, there exists a finite sequence of balls
$\{B(y_j, s_j): j=1, 2, \dots, m^\prime \}$ with $y_j \in B(0, r_0)$ and $0<s_j<r$ such that
$$
S(u)\cap \overline{B(0, r_0)} \subset \bigcup_j B(y_j, s_j ) \quad 
\text{ and } \quad
\sum_{j=1}^{m^\prime} s_j^{d-1} < C(N) r.
$$
We now fix $r=r_1$, which depends only on $d$, $\lambda$ and $N$,
 so that the second inequality in (\ref{c-2}) holds.
Let 
$$
\delta(u) =\inf\Big\{ |u(x)| +|\nabla u (x)|: x\in B(0, r_0)
\setminus \bigcup_j B(y_j, s_j) \Big\} >0.
$$
Note that if $v\in \mathcal{F}(N)$ and
$$
\| v-u\|_{C^1(B_{3/4}) }< \rho(u),
$$
where $\rho(u)>0$ is sufficiently small, then
$$
\inf\Big\{ |v(x)| +|\nabla v (x)|: x\in B(0, r_0)
\setminus \bigcup_j B(y_j, s_j) \Big\} \ge \frac12 \delta (u).
$$
We now use the fact that $\mathcal{F}(N)$ is compact with respect to the topology 
induced by the norm in $C^1(B_{3/4})$.
This implies that there exists a finite sequence of functions $\{ u_k\}_{k=1}^\ell $ in $\mathcal{F}(N)$ such that
$\mathcal{F}(N)$ is covered by the union of sets
$$
\big\{ v\in \mathcal{F}(N): \| v-u_k \|_{C^1(B_{3/4})} < \rho(u_k) \big\}.
$$
Let
$$
\delta_1=\min \big\{  \delta(u_k)/2:\  k=1, 2, \dots, \ell \big\}.
$$
Thus we have proved that for any $u\in \mathcal{F}(N)$, there exists a finite sequence of balls 
$\big\{B(y_j, s_j): j=1, 2, \dots, m^\prime \big\}$ with $y_j \in B_{1/2}$ and
$s_j\in (0, r_1)$ satisfying the second inequality in (\ref{c-2}), such that
\begin{equation}\label{c-5}
\inf\Big\{ |u(x)| +|\nabla u (x)|: x\in B(0, r_0)
\setminus \bigcup_j B(y_j, s_j) \Big\} \ge \delta_1,
\end{equation}
where $\delta_1>0$ and $r_1>0$ depends only on $d$, $\lambda$ and $N$.
This gives (\ref{c-1.5}).

The proof of (\ref{c-1}) uses a similar compactness argument and the estimate
(\ref{cd-1}). We leave the details to the reader.
\end{proof}

\begin{remark}
{\rm 
Lemma \ref{c-lemma} continues to hold if we replace the condition 
$\fint_{B_1} |u|^2\, dx  =1$ in $\mathcal{F}(N)$ by $\fint_{B_1}|u|^2\, dx \le C$.
}
\end{remark}

\begin{lemma}\label{m-lemma}
Let $F_\e (y, r)$ be defined by (\ref{F}).
Then for $r_0=\sqrt{\lambda}/4$, there exists a finite sequence of balls $\{B(y_j, s_j): j=1, 2, \dots, m^\prime\}$ such that
$y_j \in B(0, r_0)$, $s_j \in (0, r_1)$, and
\begin{equation}\label{m-1}
F_\e (0, r_0)\le C(\wN) + \frac14  \max \big\{ F_\e (y_j, s_j): j=1,2, \dots, m^\prime\big\},
\end{equation}
where $C(\wN)$ depends only on $d$, $\lambda$, $M$ and $\wN$.
\end{lemma}

\begin{proof}
By Theorem \ref{main-theorem-3}  there exists $u_0\in H^1(B_1)$ such that
$\mathcal{L}_0 (u_0)=0$ in $B_1$ and
\begin{equation}\label{m-2}
\| u_\e -u_0\|_{L^\infty(B_{3/4})}
\le C_0\, \e,
\end{equation}
\begin{equation}\label{m-3}
\fint_{B_1} |u_0|^2\, dx \le C(\wN) \fint_{B_{1/2}} |u_0|^2\, dx,
\end{equation}
\begin{equation}\label{m-3.1}
\fint_{B_1} |u_0|^2\, dx \le C_0,
\end{equation}
where we have used the assumption (\ref{m-0}).
The constant $C_0$ in (\ref{m-2})-(\ref{m-3.1}) depends only on $d$, $\lambda$ and $M$.

We now apply Lemma \ref{c-lemma}  to $u_0$.
This gives us  two sequences of balls 
$$
\big\{B(x_i, t_i), i=1,2, \dots, m\big\} \quad \text{ and } \quad
 \big\{B(y_j, s_j), j=1,2, \dots, m^\prime\big\},
 $$
with $x_i, y_j \in B(0, r_0)$ and $t_i , s_j \in (0, r_1)$, such that (\ref{c-1}), (\ref{c-1.5}) and (\ref{c-2}),
with $u_0$ in the place of $u$, hold.
We may assume that $C_0\e<\delta_1/2$.
It follows that
\begin{equation}\label{m-6}
\aligned
Z(u_\e)\cap B(0, r_0)
& \subset Z(u_\e)\cap \big\{ x\in B(0, r_0): |u_0(x)|\le C_0 \, \e \big\}\\
 & \subset \bigg( \bigcup_{i=1}^m Z(u_\e)\cap E_i\bigg) 
  \bigcup \bigg(\bigcup_{j=1}^{m^\prime} Z(u_\e)\cap B(y_j, s_j)\bigg),
\endaligned
\end{equation}
where
\begin{equation}\label{m-6.5}
E_i =\big\{ x\in B(x_i, t_i): \ |u_0(x)|\le C_0 \e \text{ and } |\nabla u_0(x)|\ge \delta_1/2\big\}.
\end{equation}
Thus,
\begin{equation}\label{m-7}
\aligned
& \mathcal{H}^{d-1} \big( Z(u_\e)\cap B(0, r_0) \big)\\
 & \le \sum_{i=1}^m \mathcal{H}^{d-1} \big( Z(u_\e)\cap E_i \big)
+\sum_{j=1}^{m^\prime} \mathcal{H}^{d-1} \big( Z(u_\e)\cap B(y_j, s_j)\big)\\
& \le   \sup_i \frac{\mathcal{H}^{d-1}  \big( Z(u_\e)\cap E_i \big)}{t_i^{d-1}}\sum_{i=1}^m t_i^{d-1}
 +  \sup_j F_\e(y_j , s_j)\sum_{j=1}^{m^\prime} s_j^{d-1}\\
&\le C (\wN) \sup_i \frac{\mathcal{H}^{d-1}  \big( Z(u_\e)\cap E_i \big)}{t_i^{d-1}}
+ \frac 14 r_0^{d-1}\sup_j  F_\e (y_j, s_j),
\endaligned
\end{equation}
where we have used (\ref{c-2}) for the last step.

Finally, note that if $t_i \le C_0 \e$, we may use Lemma \ref{s-lemma} to obtain 
$$
\aligned
\mathcal{H}^{d-1} \big( Z (u_\e)\cap E_i \big)
& \le \mathcal{H}^{d-1}\big( Z(u_\e) \cap B(x_i, t_i)\big)\\
& \le C(\wN) t_i^{d-1}.
\endaligned
$$
Otherwise,
since $\|u_0\|_{C^2(B_{3/4})}\le C$ and $|u_0 (x)|\le C_0\e$, $|\nabla u_0(x)|\ge (1/2)\delta_1$ on $E_i$,
we may cover $E_i$ by a finite sequence of balls 
$\{ B(z_k, C\e), k=1, 2, \dots, m^{\prime\prime} \}$ such that $z_k \in B(x_i, t_i)$ and
$m^{\prime\prime} \e^{d} \le Ct_i^{d-1}\e $. 
This follows readily from the Implicit Function Theorem.
By Lemma \ref{s-lemma} we see that
$$
\aligned
\mathcal{H}^{d-1} (Z(u_\e)\cap E_i)
 & \le \sum_{k=1}^{m^{\prime \prime}}
  \mathcal{H}^{d-1} \big(Z(u_\e)\cap B(z_k, C \e)\big)\\
 & \le C(\wN) m^{\prime\prime} \e^{d-1} \le C(\wN)t_i^{d-1},
 \endaligned
$$
which, together with (\ref{m-7}), completes the proof.
\end{proof}

We are now in a position to give the proof of Theorem \ref{main-theorem-1}

\begin{proof}[\bf Proof of Theorem \ref{main-theorem-1}]

To prove (\ref{nd-1}), we fix $y_0\in B(0, 2r_0)$ and $\alpha\in (0,1)$ such that
$B(y_0, 2\alpha r_0)\subset B(0, 2r_0)$.
Consider the function
$$
v(x)=u_\e (y_0+\alpha x).
$$
Then
$$
\text{\rm div} \big(\widetilde{A} (x/(\e  \alpha^{-1}))\nabla v \big) =0,
$$
where $\widetilde{A}(x)=A(x+ y_0/\e)$. Since $\widetilde{A}\in \mathcal{A}(\lambda, M)$ and
$v$ satisfies the doubling condition (\ref{ud-100}) for $|y|<\sqrt{\lambda}/4$ and $0<r\le 1/2$,
it follows by Lemma \ref{m-lemma} that
$$
\frac{\mathcal{H}^{d-1} \big( Z(v)\cap B(0, r_0)\big)}{ r_0^{d-1}}
\le C(\wN) +\frac14
\sup_{j}
\frac{\mathcal{H}^{d-1} \big( Z(v)\cap B(y_j, s_j)\big)}{ s_j^{d-1}},
$$
where $y_j\in B(0, r_0)$ and $s_j\in (0, r_1)$ for $j=1, 2, \dots, m^\prime$.
By a change of variables this leads to
$$
\frac{\mathcal{H}^{d-1} \big( Z(u_\e )\cap B(y_0, \alpha r_0)\big)}{ (\alpha r_0)^{d-1}}
\le C(\wN) +\frac14
\sup_{j}
\frac{\mathcal{H}^{d-1} \big( Z(u_\e )\cap B(y_0 +\alpha y_j , \alpha s_j )\big)}{(\alpha s_j)^{d-1}}.
$$
Thus, for any $y_0\in B(0, 2r_0)$ and $\alpha\in (0,1)$ such that $B(y_0, 2\alpha r_0)\subset B(0,2r_0)$, we have
\begin{equation}\label{m-10}
F_\e (y_0, \alpha r_0 ) \le C(\wN)+\frac14 \sup_j
F_\e (y_0+\alpha y_j , \alpha s_j)
\end{equation}
for some $y_j \in B(0, r_0)$ and $s_j \in (0, r_1)$, $j=1, 2, \dots, m$.

Finally, we iterate the estimate (\ref{m-10}) and stop  the process for $F_\e (y, r)$ whenever $r<C_0\e$.
Using  the fact that $s_j< r_1< (1/4) r_0$,
  we may deduce  that
$$
\aligned
F_\e (0, r_0)
& \le C(\wN) \sum_{k=1}^\infty 4^{-k}
+ C  \sup_{\substack{ y\in B(0, 2r_0)\\ 0<r<C_0\e }}
F_\e (y, r)\\
&\le C (\wN),
\endaligned
$$
where we have used Lemma \ref{s-lemma} for the last step.
\end{proof}

 \bibliographystyle{amsplain}
 
\bibliography{Lin-Shen-2018.bbl}

\bigskip

\begin{flushleft}

Fanghua Lin, 
Courant Institute of Mathematical Sciences, 
251 Mercer Street, 
New York, NY 10012, USA.

Email: linf@cims.nyu.edu

\bigskip

Zhongwei Shen,
Department of Mathematics,
University of Kentucky,
Lexington, Kentucky 40506,
USA.

E-mail: zshen2@uky.edu
\end{flushleft}

\bigskip

\medskip

\end{document}